\documentclass[11pt,reqno]{amsart}
\usepackage[utf8]{inputenc}
\usepackage{amsmath}
\usepackage{amsfonts}
\usepackage{amssymb}
\usepackage{amsthm}
\usepackage{appendix}
\usepackage{indentfirst}
\usepackage{mathtools}
\mathtoolsset{showonlyrefs=true}
\usepackage{enumitem}
\usepackage{hyperref}
\usepackage{geometry}

\hypersetup{
colorlinks=true,
linkcolor=blue,
citecolor=blue,
}

\title[Expansive solutions of the $N$-body problem]{Expansive solutions with prescribed asymptotics of the classical $N$-body problem}
\author{Yutong Wu}
\address{Department of Mathematics, Yale University, New Haven, CT 06511} 
\email{yutong.wu.yw894@yale.edu}

	

\def \N{\mathbb{N}}
\def \R{\mathbb{R}}
\def \C{\mathbb{C}}
\def \Z{\mathbb{Z}}
\def \d{\mathrm{d}}
\def \dt{\mathrm{d}t}
\def \X{\mathcal{X}}
\def \Y{\mathcal{Y}}
\def \B{\mathcal{B}}
\DeclareMathOperator{\im}{Im}
\renewcommand{\Re}{\mathrm{Re}}

\begin{document}

\numberwithin{equation}{section} 

\newtheorem{thm}{Theorem}
\newtheorem{lem}{Lemma}[section]
\newtheorem{prop}{Proposition}
\newtheorem*{defn}{Definition}
\newtheorem{rmk}{Remark}[section]

\begin{abstract}
We consider the classical $N$-body problem with the $\frac{1}{|x|^p}$-type potential, where $p>0$. We construct hyperbolic, parabolic and hyperbolic-parabolic solutions with prescribed asymptotic data as $t \to+\infty$.
\end{abstract}

\maketitle

\section{Introduction}

The $N$-body problem is a fundamental model in classical mechanics describing the motion of $N$ interacting point particles in $\R^d$. In this paper, we fix the dimension $d \in \N_+$, the number of bodies $N \in \N_{>1}$, the masses of the bodies $\lambda_1, \cdots, \lambda_N>0$, and a parameter $p>0$ indicating the strength of the potential. We study
\begin{equation} \label{eq N-body}
    \dot{\alpha}_i= \beta_i, \quad \lambda_i \dot{\beta}_i= -\sum_{j \neq i} \frac{\lambda_i \lambda_j(\alpha_i- \alpha_j)}{|\alpha_i- \alpha_j|^{p+2}}, \quad \forall i=1,2, \cdots, N,
\end{equation}
where $\alpha_i$'s and $\beta_i$'s take values in $\R^d$. Here, $\alpha_i$ represents the position and $\beta_i$ represents the velocity of the $i$-th body. 

We can write \eqref{eq N-body} in a more compact form. Let 
\begin{equation}
    \alpha= (\alpha_1, \cdots, \alpha_N), \;\; \beta= (\beta_1, \cdots, \beta_N), \;\; \text{and} \;\; P=(\alpha, \beta).
\end{equation} 
Set
\begin{equation}
    \X= \R^{dN}, \quad \Y= \left\{ \alpha \in \X \ \big| \ \alpha_j \neq \alpha_k, \ \forall j \neq k \right\}, \quad \Delta= \X \setminus \Y.
\end{equation}
For $\alpha \in \X$, define
\begin{equation}
    U(\alpha):= \frac{1}{p} \sum_{i<j} \frac{\lambda_i \lambda_j}{|\alpha_i- \alpha_j|^p} \quad \text{and} \quad |\alpha|^2 _{\lambda}:= \sum_{i=1}^N \lambda_i |\alpha_i|^2.
\end{equation}
We see $U(\alpha)= +\infty \iff \alpha \in \Delta$. Besides, $| \cdot |^2_{\lambda}$ induces a metric on $\X$. Let $\nabla$ denote the gradient with respect to this metric. In other words, $\nabla_i= \frac{1}{\lambda_i} \partial_{\alpha_i}$. Then \eqref{eq N-body} is equivalent to 
\begin{equation} \label{eq N-body tight form}
    \ddot{\alpha}= \nabla U(\alpha).
\end{equation}

By translation invariance, the center of mass can be assumed to be the origin: $\sum_i \lambda_i \alpha_i=0$. This is a common assumption but is unnecessary for us and will not be assumed.

The complexity of the $N$-body problem has been known since  Poincar\'e's work \cite{Poincare3body} on the $3$-body problem in the 19th century. While obtaining the precise form of the solution is impossible, a key question is how to classify solutions by their behavior as $t \to +\infty$. We will be interested in \textbf{expansive solutions}, namely those for which all mutual distances diverge to infinity:
\begin{equation}
    |\alpha_i(t)- \alpha_j(t)| \to +\infty \;\; \text{as} \;\; t \to +\infty, \;\; \forall i\neq j.
\end{equation}

Chazy's pioneering work \cite{ChazyClassification} studied the final motions of the $3$-body problem and proposed the concept of \emph{hyperbolic}, \emph{parabolic}, and \emph{hyperbolic-parabolic} solutions, which were extended to the $N$-body case later by Saari \cite{ExpansiveNbody} and Marchal-Saari \cite{StructureOfExpansive}. These early papers focus mainly on the Newtonian case $p=1$, for which \eqref{eq N-body} describes the motion of $N$ bodies subject to gravitational force, yet some results are generalizable to more general $p>0$.

The classification for any $p>0$ is as follows:
\begin{itemize}
    \item \textbf{hyperbolic:} $|\alpha_i- \alpha_j| \sim t$ \footnote{For nonnegative functions $f$ and $g$, we write $f \sim g$ if there exist $c,C>0$ such that $cf \le g \le Cf$.} as $t \to +\infty$ for each $i \neq j$;
    \item \textbf{parabolic:} $|\alpha_i- \alpha_j| \sim t^{2/(p+2)}$ as $t \to +\infty$ for each $i \neq j$;
    \item \textbf{hyperbolic-parabolic:} for each $i \neq j$, either $|\alpha_i- \alpha_j| \sim t$ or $|\alpha_i- \alpha_j| \sim t^{2/(p+2)}$ as $t \to +\infty$,  and both regimes occur. 
\end{itemize}

For hyperbolic-parabolic solutions, one can divide all bodies into \emph{clusters} so that different clusters depart from each other at the hyperbolic rate $t$, and bodies within the same cluster diverge at the parabolic rate $t^{2/(p+2)}$. 

From a modern viewpoint, large-time classification is closely related to problems in scattering theory: 
\begin{enumerate}
    \item Can one parameterize expansive solutions by asymptotic data as $t \to +\infty$ such as limiting velocities and configurations?
    \item Can one construct solutions from asymptotic data as $t \to +\infty$?
\end{enumerate}

Question (1) was studied in \cite{LongtimeScatter, ScatterClassical} in the context of scattering theory for hyperbolic dynamics, and in \cite{ParaManifold} in the language of the stable manifold of parabolic orbits. Question (2) is the main focus of this paper. Chazy \cite{ChazyClassification} partially studied this question and a systematic introduction to scattering for the $N$-body problem can be found in the monograph \cite{ScatteringBook}. \vspace{8pt}

\noindent
\textbf{Main results:}

We will construct hyperbolic, parabolic, and hyperbolic-parabolic solutions with prescribed asymptotic data as $t \to +\infty$, thereby answering Question (2) above. In particular, we will give the form of asymptotic expansions up to an $o(1)$ error. 

The map that sends the limiting velocity $v$ and the asymptotic offset $x$ to the actual solution of the form $\alpha(t)= x+ vt+ o(1)$ is called a \emph{wave operator}. If the potential is long-range, meaning $p \in (0,1]$ in our setting, then the solution is not asymptotically free, and we call the map a \emph{modified wave operator} since the asymptotics require correction terms beyond $x+vt$. Our Theorem \ref{thm hyperbolic} will construct modified wave operators for hyperbolic solutions. In the parabolic case and the hyperbolic-parabolic case, prescribed data are different. We will prescribe the central configuration and sub-leading modes in the parabolic case, and inter-cluster velocities as well as intra-cluster parabolic motion in the hyperbolic-parabolic case. More precise statements will be given as we state Theorem \ref{thm parabolic} and Theorem \ref{thm mixed}. \vspace{8pt}

\noindent
\textbf{Related works:}

Many existing constructions of expansive solutions prescribe not only asymptotic data as $t \to +\infty$ but also an initial configuration or asymptotic behavior at negative infinity. A common approach is to use variational methods, typically in a weak KAM framework; see \cite{HyperfromInitial, ParafromInitial,ParaEntire, PartiallyHyperbolic, ExistenceofNbodyproblem} for example. Other works emphasize geometric methods \cite{GeneralHyperbolic} or stable and unstable manifolds at infinity \cite{Hyperbothdirection}. We refer the reader to references therein.

In contrast, our results prescribe asymptotic data solely at $t= +\infty$ and track sub-leading terms in the resulting asymptotic expansions. The method we use will be a perturbative approach based on a fixed-point argument.

Existing constructions of this type focus mainly on the hyperbolic case. In \cite[Theorem~5.3]{ScatteringHyper}, the authors constructed hyperbolic solutions via the Dollard--M{\o}ller transform, for which a standard reference in the quantum setting is \cite{Quantumlongrange} by Dollard. The  proof of their Theorem~5.3 was provided in detail for $p>1/2$, with an indication that this restriction can be removed using arguments in \cite{ScatteringHyperArgument}. Their Theorem~5.3 is close in spirit to our Theorem \ref{thm hyperbolic} below, while our result gives the explicit form of the remainders and the uniqueness of the solution for all $p>0$. For parabolic and hyperbolic-parabolic dynamics, our Theorem \ref{thm parabolic} and Theorem \ref{thm mixed} complement existing approaches and results by constructing solutions from prescribed data at $+\infty$ with sub-leading terms using the fixed-point argument. \vspace{8pt}

\noindent
\textbf{Case 1: hyperbolic.}

\begin{thm}[hyperbolic solutions] \label{thm hyperbolic}
Take $n \in \N$ such that $\frac{1}{n+1} \le p < \frac{1}{n}$. Let $x \in \X$ and $v \in \Y$.

(i) If $p>\frac{1}{n+1}$, then there exist $c^1, \cdots, c^n \in \X$ independent of $x$ and $T_0>0$ such that \eqref{eq N-body} has a unique solution $P(t)$ for $t > T_0$ with
\begin{equation} \label{eq asymptotic hyperbolic} \begin{aligned} 
    \alpha(t) &= x+ vt+ \sum_{k=1}^n c^k t^{1-kp}+ O \big( t^{1-(n+1)p} \big) \\
    \beta(t) &= v+ \sum_{k=1}^n (1-kp) c^k t^{-kp}+ O \big( t^{-(n+1)p} \big) \end{aligned} \quad \text{as} \quad t \to +\infty.
\end{equation}

(ii) If $p=\frac{1}{n+1}$, then there exist $c^1, \cdots, c^{n+1} \in \X$ independent of $x$ and $T_0>0$ such that \eqref{eq N-body} has a unique solution $P(t)$ for $t > T_0$ with
\begin{equation} \begin{aligned} 
    \alpha(t) &= x+ vt+ \sum_{k=1}^n c^k t^{1-kp}+ c^{n+1} \log t+ O \big( t^{1-(n+2)p} \log t \big) \\
    \beta(t) &= v+ \sum_{k=1}^n (1-kp) c^k t^{-kp}+ \frac{c^{n+1}}{t}+ O \big( t^{-(n+2)p} \log t \big) \end{aligned} \quad \text{as} \quad t \to +\infty.
\end{equation}
\end{thm}

In the short-range case $p>1$, it is well-known that hyperbolic solutions are free: $\alpha(t)=x+vt+o(1)$, which is covered by our result. We also recover the Newtonian case $p=1$ where one knows $\alpha(t)= x+ vt+ c\log t+ o(1)$ with $c$ only depending on $v$. If $\frac{1}{2}<p<1$, then we need one more additional term, which was explicitly studied in \cite{ScatteringHyper}. Our result achieves its full strength when $p>0$ is small. \vspace{8pt}

\noindent
\textbf{Case 2: parabolic.}

We say $a \in \Y$ is \textbf{minimal} if for any $\tilde{a} \in \Y$,
\begin{equation}
    |\tilde{a}|_{\lambda}= |a|_{\lambda} \implies U(\tilde{a}) \ge U(a).
\end{equation}
We say $a \in \Y$ is \textbf{normal}\footnote{In some literature, people say $a$ is a \emph{minimal normalized central configuration} if $a$ is minimal as defined above, and $|a|_{\lambda}=1$, $\sum_i \lambda_i a_i=0$. Here, ``normalized'' refers to the condition $|a|_{\lambda}=1$. However, in this paper, it would be more convenient to normalize $a$ so that $at^{2/(p+2)}$ solves \eqref{eq N-body}, so we adopt the above definition of ``normal''.} if $a$ is minimal and $P(t)$ defined by
\begin{equation}
    \alpha(t)= a t^{\frac{2}{p+2}}, \quad \beta(t)= \dot{\alpha}(t)
\end{equation}
is a solution of \eqref{eq N-body} for $t>0$. 

\begin{rmk}
If $a \in \Y$ is minimal, then there exists a unique $h>0$ such that $ha$ is normal. This can be deduced from the method of Lagrange multipliers.
\end{rmk}

For $\gamma \in \R$, define the following space of functions in $t$ on $(0,+\infty)$:
\begin{equation}
    S_{\gamma}:= \left\{ \sum_{i=1}^m y_i t^{q_i} \log^{k_i} (t) \ \Big| \ m \in \N,\ y_i \in \X,\ q_i \le \gamma,\ k_i \in \N \right\}.
\end{equation}

\begin{thm}[parabolic solutions] \label{thm parabolic} Assume $p \in (0,2)$. Let $a \in \Y$ be normal. Then there exists a basis $\{b^1, \cdots, b^{dN}\}$ of $\X$ and $p_1, \cdots, p_{dN} \le \frac{p}{p+2}$ satisfying: for any $x=(x_1, \cdots, x_{dN}) \in \X$, there exists $\delta>0$, $T_0>0$, a solution $P_x(t)$ of \eqref{eq N-body} for $t>T_0$ and $f \in S_{\frac{2p-2}{p+2}}$ such that 
\begin{equation} \label{eq asymptotic parabolic}
    \alpha_x(t) = at^{\frac{2}{p+2}}+ \sum_{k=1}^{dN} x_k b^k t^{p_k}+ f(t)+ o(t^{-\delta+})\footnote{For a function $g$ of $t$ and $\mu \in \R$, we write $g(t)=o(t^{\mu+})$ as $t \to +\infty$ if $\lim\limits_{t \to +\infty} \frac{g(t)} {t^{\mu+ \epsilon}}= 0$ for all $\epsilon>0$.} \quad \text{as} \quad t \to +\infty.
\end{equation}
Moreover, for a fixed $a$, the map $x \mapsto P_x$ is injective.

In particular, if $p \in (0,1)$, we have
\begin{equation} \label{eq asymptotic parabolic p<1}
    \alpha_x(t) = at^{\frac{2}{p+2}}+ \sum_{k=1}^{dN} x_k b^k t^{p_k}+ o \big( t^{-\frac{2-2p}{p+2}+} \big) \quad \text{as} \quad t \to +\infty.
\end{equation}
\end{thm}

\begin{rmk}
We may have $p_k< 0$ for some $k$, so it does not follow directly from \eqref{eq asymptotic parabolic} that the map $x \mapsto P_x$ is injective, and $P_x$ cannot be unique for the same reason.
\end{rmk}

The modes $b^k t^{p_k}$ are solutions to a linearized equation. We need the assumption $p \in (0,2)$ to ensure that there are $dN$ such modes. See Remark \ref{rmk para assumption}. \vspace{8pt}

\noindent
\textbf{Case 3: hyperbolic-parabolic.}

For $v \in \X$, the relation $i \sim j$ given by $v_i= v_j$ is an equivalence on $\{1, \cdots, N\}$. Let $M_v$ denote the set of equivalence classes. Each element in $M_v$ is a cluster and we thus obtain a cluster partition among all bodies.

\begin{thm}[hyperbolic-parabolic solutions] \label{thm mixed}
Assume $v \in \X$ and $P^0 \in C^1 \big( \R_+, \X \times \X \big)$ satisfy: (1) for any $I \in M_v$, $\{P_i^0 \ | \ i \in I\}$ solves the $|I|$-body problem, that is, 
\begin{equation} \label{eq sub N-body}
    \dot{\alpha}_i^0= \beta_i^0, \quad \lambda_i \dot{\beta}_i^0= -\sum_{j \in I \setminus \{i\}} \frac{\lambda_i \lambda_j (\alpha_i^0- \alpha_j^0)}{|\alpha_i^0- \alpha_j^0|^{p+2}}, \quad \forall i \in I;
\end{equation}
(2) there exists $a \in \X$ with $a_i \neq a_j$, $\forall i \sim j, i \neq j$ and $q< \min \big\{\frac{2}{p+2}+ p-1, \frac{2}{p+2} \big\}$ such that
\begin{equation} \label{eq parabolic assumption}
    \alpha^0(t)= a t^{\frac{2}{p+2}}+ O \big( t^q \big) \quad \text{as} \quad t \to +\infty.
\end{equation}
Then there exists $\delta>0$, $T_0>0$, a solution $P(t)$ of \eqref{eq N-body} for $t>T_0$ and $f \in S_{1-p}$ such that 
\begin{equation} \label{eq asymptotic mixed}
    \alpha(t)= vt+ \alpha^0(t)+ f(t)+ o \big( t^{-\delta+} \big) \quad \text{as} \quad t \to +\infty.
\end{equation}

In particular, if $p>1$, we have
\begin{equation}
    \alpha(t)= vt+ \alpha^0(t)+ o \big( t^{1-p+} \big) \quad \text{as} \quad t \to +\infty.
\end{equation}
\end{thm}

\begin{rmk}
If $p \in (0,2)$, then all the parabolic solutions constructed in Theorem \ref{thm parabolic} satisfy \eqref{eq parabolic assumption} by taking $q= \frac{p}{p+2}$.
\end{rmk}

\begin{rmk}
If a singleton $\{i\} \in M_v$, then the assumption implies $\beta_i^0=0$ and $\alpha_i^0$ is constant. If $v \in \Y$, then $M_v$ consists of singletons and Theorem \ref{thm mixed} can be understood as a generalization of Theorem \ref{thm hyperbolic}, while Theorem \ref{thm hyperbolic} gives more explicit expressions and asserts uniqueness.
\end{rmk} 

\noindent 
\textbf{Comments on the results:}

(1) Our results construct families of solutions of expected dimensions. Compare our results to \cite{ExistenceofNbodyproblem} for instance. For any given $v \in \Y$ in the hyperbolic case, and for any given normal $a$ in the parabolic case, they constructed solutions for any initial configuration, while we constructed solutions for any $x \in \X$; there are both $dN$ free parameters after fixing $v$ or $a$. In the hyperbolic-parabolic case, for fixed $v \in \X$ and central configuration $a$ compatible with the cluster partition for $v$, they again constructed solutions for any initial configuration, while we constructed solutions for any intra-cluster parabolic solutions with leading coefficient $a$; if $p \in (0,2)$, this again yields $dN$ free parameters in total according to Theorem \ref{thm parabolic}. 

(2) An application of our theorems in PDE and in the quantum context is the construction of multi-soliton solutions of Hartree equations. Due to the non-local nature of Hartree equations, the soliton centers follow the $N$-body dynamics instead of free motion. The construction of multi-soliton solutions for Hartree equations was first done by \cite{KMR2bodyhartree} for two solitons and was extended to the $N$-soliton case recently by \cite{Hartree3D, Hartree4D, Hartree4Dmixed}, including hyperbolic, parabolic, and hyperbolic-parabolic regimes for the center dynamics. Our results provide a large class of $N$-body trajectories with prescribed asymptotic data at $t=+\infty$, and therefore yield large families of corresponding multi-soliton solutions to Hartree; moreover, the explicit asymptotic expansions established here translate into refined expansions for the multi-soliton solutions. \vspace{8pt}

\noindent
\textbf{Strategy of the proof:} 

The proofs of the three theorems all start with constructing some approximate solution $\tilde{\alpha}(t)$ such that
\begin{equation} \label{eq eq of tilde alpha}
    \ddot{\tilde{\alpha}}(t)= \nabla U(\tilde{\alpha})+ o(t^{-2}).
\end{equation}
Then we take a perturbative perspective and turn to solve 
\begin{equation} \label{eq perturbe tilde alpha}
    \ddot{\alpha}- \ddot{\tilde{\alpha}}= \nabla U(\alpha)- \nabla U(\tilde{\alpha})+ o(t^{-2}).
\end{equation}

For Theorem \ref{thm hyperbolic}, we can construct $\tilde{\alpha}$ explicitly so that \eqref{eq eq of tilde alpha} holds, and \eqref{eq perturbe tilde alpha} is easy to deal with because the right hand side decays fast enough. 

For Theorem \ref{thm parabolic} and Theorem \ref{thm mixed}, $\tilde{\alpha}$ cannot be written out explicitly, but one may expand $\nabla U$ by the Taylor formula and construct $\tilde{\alpha}$ by induction. In each step, we add some term to $\tilde{\alpha}$ that is in the space $S_\gamma$ to cancel out the highest order term left in the expansion. Since more terms may appear because a new term is added to $\tilde{\alpha}$, we provide an argument showing that after finitely many steps, all remaining terms will be of order less than a fixed threshold. This part, as well as the analysis of \eqref{eq perturbe tilde alpha}, requires studying an equation of the form 
\begin{equation}
    \ddot{z}(t)= \frac{Az(t)}{t^2}+ F(t,z(t)),
\end{equation}
where $z$ is a vector function in $t$, $A$ is a constant matrix, and $F$ has a suitable decay. We prove the existence of solutions to such equation in Lemma \ref{lem ODE}. We remark that a straightforward fixed point argument would not work, because $\frac{Az(t)}{t^2}$ always has the same order of decay as $\ddot{z}(t)$. We overcome this by constructing the Green function of the operator $\frac{\d^2}{\d t^2}- \frac{A}{t^2}$. In fact, $b^k$ and $p_k$ in Theorem \ref{thm parabolic} are chosen so that $b^k t^{p_k}$ is in the kernel of this operator.

\section{Hyperbolic solutions}

\begin{proof}[Proof of Theorem \ref{thm hyperbolic}]\

We first consider case (i). We begin with the following lemma.

\begin{lem} \label{lem expansion hyperbolic}
For $x \in \X$ and $v \in \Y$, there exist $c^1, \cdots, c^n \in \X$ only depending on $v$ such that 
\begin{equation} \label{eq expansion of nabla U}
    \nabla U(\alpha^n)- \ddot{\alpha}^n= O \big( t^{-1-(n+1)p} \big) \quad \text{as} \quad t \to +\infty,
\end{equation}
where 
\begin{equation}
    \alpha^n= x+ vt+ \sum_{k=1}^n c^k t^{1-kp}.
\end{equation}
\end{lem}

\begin{proof}
Since $v \in \Y$, when $t$ is large enough, we have $\alpha^n \in \Y$. This means $\nabla U(\alpha^n)$ is well defined. Now \eqref{eq expansion of nabla U} is equivalent to
\begin{equation}
    t^{1+p} \nabla U(\alpha^n)= \sum_{k=1}^n kp(kp-1) c^k t^{-(k-1)p}+ O \big( t^{-np} \big).
\end{equation}
We set $b_{ij}^0= v_i- v_j$ and $x_{ij}= x_i- x_j$, which are given. We also let $b^k= kp(kp-1) c^k$, $b_{ij}^k= c^k_i- c^k_j$. Note that $kp(kp-1) \neq 0$, so $b^k$ would determine $c^k$ and thus $b_{ij}^k$ for all $i,j$. It then suffices to find $b^1, \cdots, b^n$ such that (we write $b^k=(b_1^k, \cdots, b_N^k)$ with $b_i^k \in \R^d$) 
\begin{equation}
    -\sum_{j \neq i} \lambda_j\frac{\sum \limits_{k=0}^l b_{ij}^k t^{-kp}+ x_{ij} t^{-1}}{\bigg| \sum \limits_{k=0}^l b_{ij}^k t^{-kp}+ x_{ij} t^{-1} \bigg|^{p+2}}= \sum_{k=1}^l b_i^k t^{-(k-1)p}+ O \big( t^{-lp} \big), \quad \forall i,\ \forall 1 \le l \le n.
\end{equation}
Note $b_{ij}^0 \neq 0$ if $i \neq j$. Since $lp \le np<1$, the above is further equivalent to 
\begin{equation} \label{eq taylor formula}
    -\sum_{j \neq i} \frac{\lambda_j \sum \limits_{k=0}^l b_{ij}^k t^{-kp}}{\bigg| \sum \limits_{k=0}^l b_{ij}^k t^{-kp} \bigg|^{p+2}}= \sum_{k=1}^l b_i^k t^{-(k-1)p}+ O \big( t^{-lp} \big), \quad \forall i,\ \forall 1 \le l \le n.
\end{equation}

We will determine $b^1, \cdots, b^n$ by induction. We choose $b^1$ by setting 
\begin{equation}
    b_i^1= -\sum_{j \neq i} \frac{\lambda_j b_{ij}^0}{|b_{ij}^0|^{p+2}}, \quad \forall i. 
\end{equation}
Then \eqref{eq taylor formula} holds for $l=1$. Suppose $b^1, \cdots, b^{s-1}$ are chosen such that \eqref{eq taylor formula} holds for $1 \le l \le s-1$. Then the LHS of \eqref{eq taylor formula} with $l=s-1$ is determined and is a smooth function of $t^{-p}$ when $t$ is large, since $b_{ij}^0 \neq 0$. By the Taylor formula, there exist ${b^1}', \cdots, {b^s}' \in \X$ such that
\begin{equation}
    -\sum_{j \neq i} \frac{\lambda_j\sum \limits_{k=0}^{s-1} b_{ij}^k t^{-kp}}{\bigg| \sum \limits_{k=0}^{s-1} b_{ij}^k t^{-kp} \bigg| ^{p+2}}= \sum_{k=1}^s {b^k_i}' t^{-(k-1)p}+ O \big( t^{-sp} \big), \quad \forall i.
\end{equation}
We want to show ${b^k}'=b^k$ for $1 \le k \le s-1$. In fact, the above formula implies
\begin{equation}
    -\sum_{j \neq i} \frac{\lambda_j\sum \limits_{k=0}^{s-1} b_{ij}^k t^{-kp}}{\bigg| \sum \limits_{k=0}^{s-1} b_{ij}^k t^{-kp} \bigg| ^{p+2}}= \sum_{k=1}^{s-1} {b^k_i}' t^{-(k-1)p}+ O \big( t^{-(s-1)p} \big), \quad \forall i.
\end{equation}
But applying \eqref{eq taylor formula} with $l=s-1$ yields 
\begin{equation}
    -\sum_{j \neq i} \frac{\lambda_j \sum \limits_{k=0}^{s-1} b_{ij}^k t^{-kp}}{\bigg| \sum \limits_{k=0}^{s-1} b_{ij}^k t^{-kp} \bigg| ^{p+2}}= \sum_{k=1}^{s-1} b^k_i t^{-(k-1)p}+ O \big( t^{-(s-1)p} \big), \quad \forall i.
\end{equation}
The above two formulas then force ${b^k}'=b^k$ for $1 \le k \le s-1$. We take $b^s= {b^s}'$ and the induction process can be continued.

Also, from the proof we see $b^k$ does not depend on $x$. The lemma is proved.
\end{proof}

Go back to the proof of Theorem \ref{thm hyperbolic}. Let $P^n= (\alpha^n, \beta^n)$, where $\beta^n= \dot{\alpha}^n$ and $\alpha^n$ is as in Lemma \ref{lem expansion hyperbolic}. Let $\B= \Big\{ P \in C^0 \big( [T_0, +\infty); \X \times \X \big) \ \big| \ {\| P- P^n \|}_1 \le R \Big\}$, where
\begin{equation}
    {\| P \|}_1:= \sup_{t \ge T_0} \Big( \epsilon t^{(n+1)p-1} |\alpha(t)|+ t^{(n+1)p} |\beta(t)| \Big),
\end{equation}
and $\epsilon>0$, $R>0$, $T_0>0$ are to be chosen later.

Firstly, if $P \in \B$, then
\begin{equation}
    \alpha(t)= \alpha^n(t)+ O(t^{1-(n+1)p})= vt+ O(t^{1-p}) \;\; \text{as} \;\; t\to \infty.
\end{equation}
Then for $T_0$ large enough, we have $\alpha(t) \in \Y$ for all $t \ge T_0$. This ensures that $\nabla U(\alpha)$ is well defined for $P \in \B$. We define $\Gamma P= (\Gamma \alpha, \Gamma \beta)$ by
\begin{align}
    \Gamma \alpha(t) &= \alpha^n(t)- \int_t^\infty (\beta(\tau)- \beta^n(\tau)) \d \tau, \\
    \Gamma \beta(t) &= \beta^n(t)- \int_t^\infty \big( \nabla U(\alpha(\tau))- \dot{\beta}^n(\tau) \big) \d \tau.
\end{align}
We claim that: if $\epsilon^{-1}, R, T_0$ are large enough, then $\Gamma$ maps $\B$ to $\B$ and is a contraction.

In fact, assume $P \in \B$. Then
\begin{equation}
    |\Gamma \alpha(t)- \alpha^n(t)| \le \int_t^\infty |\beta(\tau)- \beta^n(\tau)| \d \tau \le R \int_t^\infty \tau^{-(n+1)p} \d \tau \lesssim R t^{1-(n+1)p},
\end{equation}
and by the Taylor formula and Lemma \ref{lem expansion hyperbolic},
\begin{align}
    |\Gamma \beta(t)- \beta^n(t)| &\le \int_t^\infty \Big( \big| \nabla U(\alpha(\tau))- \nabla U(\alpha^n(\tau)) \big|+ \big| \nabla U(\alpha^n(\tau))- \dot{\beta}^n(\tau) \big| \Big) \d \tau \\
    &\lesssim \int_t^\infty \big( \tau^{-2-p} \epsilon^{-1} R \tau^{-(n+1)p+1}+ \tau^{-1-(n+1)p} \big) \d \tau \\
    &\lesssim \epsilon^{-1} R t^{-1-p}+ t^{-(n+1)p}.
\end{align}
Here implicit constants only depend on $p$, $x$ and $v$. By taking $\epsilon$ small enough, $R$ large enough and $T_0$ large enough in order, we get $\Gamma P \in \B$. By the same calculation we can show $\Gamma$ is a contraction provided $\epsilon^{-1}, R, T_0$ are large enough. Then the contraction mapping theorem states that $\Gamma$ has a unique fixed point in $\B$.

Since $\frac{\d}{\dt} \Gamma \alpha= \beta$ and $\frac{\d}{\dt} \Gamma \beta= \nabla U(\alpha)$, we see the fixed point of $\Gamma$ in $\B$ solves \eqref{eq N-body} for $t>T_0$. Now suppose $P$ solves \eqref{eq N-body} for $t>T_0$ and satisfies \eqref{eq asymptotic hyperbolic}. Then 
\begin{equation}
    \lim_{t \to +\infty} |\alpha(t)- \alpha^n(t)|= \lim_{t \to +\infty} |\beta(t)- \beta^n(t)|= 0.
\end{equation}
Thus by the fundamental theorem of calculus, we have
\begin{equation}
    \Gamma \alpha(t)= \alpha^n(t)- \int_t^\infty \big( \dot{\alpha}(\tau)- \dot{\alpha}^n(\tau) \big) \d \tau= \alpha^n(t)+ \alpha(t)- \alpha^n(t)= \alpha(t)
\end{equation}
and similarly $\Gamma \beta(t)= \beta(t)$. Finally, note that $P$ satisfies \eqref{eq asymptotic hyperbolic} $\iff$ $P \in \B$ for some $\epsilon>0, R>0$ and large enough $T_0>0$. We conclude the proof of case (i).

For case (ii), using the proof of Lemma \ref{lem expansion hyperbolic}, we can prove that: for $x \in \X$ and $v \in \Y$, there exist $c^1, \cdots, c^{n+1} \in \X$ only depending on $v$ such that 
\begin{equation}
    \nabla U(\alpha^n)- \ddot{\alpha}^n= O \big( t^{-1-(n+2)p} \log t \big) \quad \text{as} \quad t \to +\infty, 
\end{equation}
where
\begin{equation}
    \alpha^n= x+ vt+ \sum_{k=1}^n c^k t^{1-kp}+ c^{n+1} \log t.
\end{equation}
We replace the norm ${\| \cdot \|}_1$ by ${\| \cdot \|}_2$ defined as
\begin{equation}
    {\| P \|}_2:= \sup_{t \ge T_0} \Big( \frac{\epsilon t^{(n+2)p-1}}{\log t} |\alpha(t)|+ \frac{t^{(n+2)p}}{\log t} |\beta(t)| \Big),
\end{equation}
and define $\mathcal{B}$ and $\Gamma$ as above. We can obtain the result by the same argument.
\end{proof}

\section{Parabolic solutions}

As mentioned briefly in the introduction, we need to apply the Taylor formula for $\nabla U$. Before diving into the proof, we clarify our notation for doing calculus on $\X$ under the metric induced by $|\cdot|_\lambda^2$, especially when we establish the Taylor formula.

For $x \in \X$, we write $x=(x_1, \cdots, x_N)^T$, where $x_i \in \R^d$. We equip $\X$ with the inner product
\begin{equation}
    (x,y)_\lambda := \sum_{i=1}^N \lambda_i x_i \cdot y_i \quad \text{and} \quad |x|_\lambda^2 := (x,x)_\lambda,
\end{equation}
where the dot product $\cdot$ represents Euclidean inner product. Since $\X= \R^{dN}$ is a vector space, for any $x \in \X$, we may equate the tangent space $T_x \X$ with $\X$ itself. Then this inner product induces a Riemannian metric on $\X$. We abuse the notation to let $|\cdot|_\lambda^2$ denote the metric. Let
\begin{equation}
    \Lambda= \mathrm{diag}(\lambda_1 I_d, \lambda_2 I_d, \cdots, \lambda_N I_d) \in \R^{dN \times dN}.
\end{equation}
Then $(x,y)_\lambda= x^T \Lambda y$ as standard matrix multiplication in standard coordinates.

For differentiation, the general principle is to let $D$ denote standard ones and to let $\nabla$ denote those under $|\cdot|_\lambda^2$.

Let $f$ be a function on $\X$. Let $Df$ and $D^2f$ represent standard Euclidean differential and Euclidean Hessian, and naturally realize them as vector and matrix, respectively. We mark that both $Df$ and the elements in $\X$ should be realized as column vectors whenever it matters. 

We let $\nabla f$ be the gradient of $f$ with respect to $|\cdot|_\lambda^2$. This means for $x \in \X$, $\nabla f(x)$ is a tangent vector in $T_x \X$, or a $(1,0)$-tensor, and for any $v \in T_x \X$, 
\begin{equation}
    (\nabla f(x), v)_\lambda= D f(x) \cdot v.
\end{equation}
If we realize $\nabla f(x)$ as a column vector under the standard basis, then \begin{equation}
    \nabla f(x)= \Lambda^{-1} Df(x).
\end{equation} 
This is how we understand \eqref{eq N-body tight form} and agrees with the definition there. 

For $s \in \Z_{\ge 1}$, we let $\nabla^{s+1} f$ denote the $(1,s)$-tensor obtained by applying the Levi--Civita connection $\nabla$ to the gradient vector $\nabla f$ for $s$ times.\footnote{We follow the common convention that $\nabla$ denotes both the gradient and the connection, so $\nabla^{s+1}f$ denotes the $s$-fold covariant derivative of $\nabla f$, which is a $(1,s)$-tensor. Some authors write $\mathrm{grad}(f)$ for the gradient to avoid confusion in geometry; we keep $\nabla f$ to match standard celestial mechanics notation.} Then for $x \in \X$ and $v \in T_x \X$, $\nabla^{s+1} f(x)$ is a $(1,s)$-tensor and $v^{\otimes s}$ is an $(s,0)$-tensor. We let $\nabla^{s+1}f(x)(v^{\otimes s})$ be the $(1,0)$-tensor obtained by contracting the $s$ covariant indices of $\nabla^{s+1}f(x)$ and the $s$ contravariant indices of $v^{\otimes s}$.

Using this notation, the Taylor formula for $\nabla f$ is as follows. For $x, y \in \X$, 
\begin{equation}
    \nabla f(y)= \nabla f(x)+ \sum_{s=1}^n \frac{1}{s!} \nabla^{s+1} f(x) ((y-x)^{\otimes s})+ O(|x-y|_\lambda^{n+1}).
\end{equation}
In practice, the function to which we will apply the Taylor formula is $\nabla U$, which is only defined on $\Y$. The above applies when the line segment connecting $x$ and $y$ is contained in $\Y$.

In particular, $\nabla^2 f(x)$ is a $(1,1)$-tensor, and we realize it as a matrix $A \in \R^{dN \times dN}$ so that for any $v \in T_x \X$, $\nabla^2 f(x)(v)$ is equal to the standard matrix multiplication $Av$. More concretely, we take
\begin{equation}
    A= \Lambda^{-1} D^2f(x).
\end{equation}
Note that the Euclidean Hessian matrix $D^2f(x)$ is symmetric, and $\nabla^2 f(x)$ as a $(1,1)$-tensor is symmetric with respect to $(\cdot, \cdot)_\lambda$, but the matrix realization $A$ is not symmetric.

\begin{proof}[Proof of Theorem \ref{thm parabolic}]\

Recall that $a$ is the given vector in $\Y$ that is normal. Let $A$ be the matrix representation of $\nabla^2 U(a)$ in $\R^{dN \times dN}$ as defined above. We need a lower bound of the eigenvalues of $A$.

\begin{lem} \label{lem linear algebra}
The matrix $A$ is diagonalizable and each eigenvalue is real and $\ge -\frac{2p}{(p+2)^2}$.
\end{lem}
\begin{proof}
Recall that $\Lambda A= D^2U(a)$ is symmetric. Since $\Lambda$ is a diagonal matrix with positive diagonal entries, we can define $\Lambda^{-\frac{1}{2}}$, which is also diagonal. We have
\begin{equation}
    \Lambda^{\frac{1}{2}} A \Lambda^{-\frac{1}{2}}= \Lambda^{-\frac{1}{2}} (\Lambda A) \Lambda^{-\frac{1}{2}},
\end{equation}
so $A$ is similar to $\Lambda^{-\frac{1}{2}} (\Lambda A) \Lambda^{-\frac{1}{2}}$, which is symmetric. Thus, $A$ is diagonalizable (over $\R$). 

We know $A$ and $\Lambda^{-\frac{1}{2}} (\Lambda A) \Lambda^{-\frac{1}{2}}$ have the same set of eigenvalues. Since
\begin{equation}
    \Lambda^{-\frac{1}{2}} (\Lambda A) \Lambda^{-\frac{1}{2}}= \Lambda^{-\frac{1}{2}} \Big( \Lambda A+ \frac{2p}{(p+2)^2} \Lambda \Big) \Lambda^{-\frac{1}{2}}- \frac{2p}{(p+2)^2}I
\end{equation}
and $\Lambda^{-\frac{1}{2}} \big( \Lambda A+ \frac{2p}{(p+2)^2} \Lambda \big) \Lambda^{-\frac{1}{2}}$ is congruent to $\Lambda A+ \frac{2p}{(p+2)^2} \Lambda$, it suffices to prove
\begin{equation} \label{eq lower bound of A}
    \Lambda A+ \frac{2p}{(p+2)^2} \Lambda \ge 0.
\end{equation}

Let $f(a):= U(a) |a|_{\lambda}^p$. Since $a$ is minimal and $U$ is homogeneous of degree $-p$, we see $a$ is a minimizer of $f$ on $\Y$. This implies $D^2f(a) \ge 0$. We compute
\begin{equation}
    Df(a)= D U(a) |a|_\lambda^p+ pU(a) |a|_\lambda^{p-2} \Lambda a
\end{equation}
and thus
\begin{equation}
    D^2 f(a)= D^2 U(a) |a|_\lambda^p+ 2p |a|_\lambda^{p-2} D U(a) \otimes \Lambda a+ p(p-2) U(a) |a|_\lambda^{p-4} \Lambda a \otimes \Lambda a+ pU(a) |a|_\lambda^{p-2} \Lambda.
\end{equation}
Since $a$ is normal, we have $\frac{\d^2}{\dt^2} (at^{\frac{2}{p+2}})= \nabla U(at^{\frac{2}{p+2}})$, which yields $D U(a)= -\frac{2p}{(p+2)^2} \Lambda a$. By homogeneity of $U$, we have $DU(a) \cdot a= -pU(a)$, so $U(a)= \frac{2}{(p+2)^2} |a|_\lambda^2$. We then get
\begin{equation}
    D^2 f(a) = |a|_\lambda^p \Lambda A- \frac{2p}{p+2} |a|_\lambda^{p-2} \Lambda a \otimes \Lambda a+ \frac{2p}{(p+2)^2} |a|_\lambda^p \Lambda \le |a|_\lambda^p \Lambda A+ \frac{2p}{(p+2)^2} |a|_\lambda^p \Lambda.
\end{equation}
Then we obtain \eqref{eq lower bound of A} in view of $D^2 f(a) \ge 0$. This implies the lower bound of the eigenvalues.
\end{proof}

From Lemma \ref{lem linear algebra}, there are eigenvectors $b^1, \cdots, b^{dN}$ of $A$ forming a basis of $\X$, with eigenvalues $c_1, \cdots, c_{dN}$, respectively, and $c_i \ge -\frac{2p}{(p+2)^2}$. Then there exists $p_i \le \frac{p}{p+2}$ such that $p_i^2- p_i= c_i$ for each $i$. This construction ensures $\frac{\d^2}{\dt^2} b^i t^{p_i}=A b^i t^{p_i}$. We may assume $p_1 \ge \cdots \ge p_{dN}$. 

\begin{rmk} \label{rmk para assumption}
If $p \in (0,2)$, then we have $p_i< \frac{2}{p+2}$, the latter being the self-similar exponent. This makes $t^{p_i}$ sub-leading terms, and thus they are allowed in the expansion. 
\end{rmk}

\;\\
\textbf{Case 1:} $p \in (0,1)$.

This case is easier to deal with and we have more explicit expressions.

Let $x \in \X$. We will inductively construct $P_x^k(t)= P^k(t)$ for $1 \le k \le dN$ with each $P^k(t)$ being a solution of \eqref{eq N-body} for large $t$, and for all $k=1, 2, \cdots, dN$,
\begin{equation} \label{eq induction asymptotic}
    \alpha^k(t)= \alpha^{k-1}(t)+ x_k b^k t^{p_k}+ o \big( t^{-\frac{2-p}{p+2}+ p_k+} \big) \quad \text{as} \quad t \to +\infty,
\end{equation}
where we set $\alpha^0(t)= at^{\frac{2}{p+2}}$ and $\beta^0(t)= \dot{\alpha}^0(t)$. 

Suppose $P^{k-1}(t)$ is constructed for some $k \ge 1$ with required properties. Let
\begin{equation}
    \alpha^k(t) = \alpha^{k-1}(t)+ x_k b^k t^{p_k}+ z(t)
\end{equation}
and assume $|z(t)| \le t^{p_k}$. For any position $\tilde{\alpha}$ on the segment connecting $\alpha^k(t)$ and $\alpha^{k-1}(t)$, 
\begin{equation}
    \tilde{\alpha}(t)= at^{\frac{2}{p+2}}+ O(t^{\frac{p}{p+2}}).
\end{equation}
This means $\tilde{\alpha}(t) \in \Y$, thus we can apply the Taylor formula  to get
\begin{align}
    \nabla U(\alpha^k)- \nabla U(\alpha^{k-1})= \nabla^2 U(\alpha^{k-1}) (\alpha^k- \alpha^{k-1})+ O \big( t^{-2-\frac{2}{p+2}+ 2p_k} \big)
\end{align}
using $\alpha^k- \alpha^{k-1}= O(t^{p_k})$. Since $\alpha^{k-1}= at^{\frac{2}{p+2}}+ O(t^\frac{p}{p+2})$, we further have
\begin{equation}
    \nabla^2 U(\alpha^{k-1}) (\alpha^k- \alpha^{k-1})= \frac{A (\alpha^k- \alpha^{k-1})}{t^2}+ O(t^{-2- \frac{2-p}{p+2}+ p_k}).
\end{equation}
Note that $p_k \le \frac{p}{p+2}$. Set $\delta= \frac{2-p}{p+2}- p_k>0$. Then we get
\begin{equation}
    \nabla U(\alpha^k)- \nabla U(\alpha^{k-1})= \frac{A (\alpha^k- \alpha^{k-1})}{t^2}+ O(t^{-2-\delta}).
\end{equation}
By similar calculations, the partial derivative in $z$ of the $O(t^{-2-\delta})$ term is $O(t^{-2-\frac{2-p}{p+2}})$.

Note that $y(t):= x_k b^k t^{p_k}$ solves $\ddot{y}(t)= \frac{A y(t)}{t^2}$. The equation $\ddot{\alpha}^k= \nabla U(\alpha^k)$ then becomes 
\begin{equation} \label{eq ODE1}
    \ddot{z}(t)= \frac{A z(t)}{t^2}+ F(t,z),
\end{equation}
where $F$ satisfies
\begin{equation}
    \sup_{|z| \le t^{p_k}} |F(t,z)| \lesssim t^{-2-\delta}, \quad \sup_{|z| \le t^{p_k}} |\nabla_z F(t,z)| \lesssim t^{-2-\frac{2-p}{p+2}}, \quad \forall t>0.
\end{equation}

\begin{lem} \label{lem ODE}
Let $\delta>0$, $\kappa>0$ and $\epsilon< \delta$. Assume $B \in \R^{dN \times dN}$ is diagonalizable over $\C$ and $F: \R_+ \times \X \to \X$ satisfies
\begin{equation} \label{eq estimate of F}
    \sup_{|z| \le t^{-\epsilon}} |F(t,z)| \lesssim t^{-2-\delta}, \quad \sup_{|z| \le t^{-\epsilon}} |\nabla_z F(t,z)| \lesssim t^{-2-\kappa}, \quad \forall t>0.
\end{equation}
Then there exists $z(t)= o(t^{-\delta+})$ as $t \to +\infty$ such that the following holds for large $t$:
\begin{equation} \label{eq ODE with decay}
    \ddot{z}(t)= \frac{Bz(t)}{t^2}+ F(t,z(t)).
\end{equation}
\end{lem}

This lemma is similar to \cite[Lemma~5.2]{Hartree3D}. We will give the proof in the appendix.  

By Lemma \ref{lem linear algebra}, we can apply Lemma \ref{lem ODE} to solve \eqref{eq ODE1}, then we can find $P^k$ with the required properties.

Finally, we set $P_x(t)= P_x^{dN}(t)$. Since $p_k \le \frac{p}{p+2}$, $\forall k$, we can derive \eqref{eq asymptotic parabolic p<1} from \eqref{eq induction asymptotic}. It remains to show injectivity. Suppose $P_x(t)= P_y(t)$ and $x \neq y \in \X$. Take the minimal $l \in \{1, \cdots, dN\}$ such that $x_l \neq y_l$. Then we have $P_x^{l-1}(t)= P_y^{l-1}(t)$ by construction. Adding up \eqref{eq induction asymptotic}, we get 
\begin{equation}
    \sum_{k=l}^{dN} (x_k-y_k) b^k t^{p_k}= o \big( t^{-\frac{2-p}{p+2}+ p_l+} \big).
\end{equation}
Take the maximal $u \ge l$ such that $p_u= p_l$. Then we have 
\begin{equation}
    \sum_{k=l}^u (x_k- y_k) b^k t^{p_l}= o(t^{p_l}),
\end{equation}
which means $\sum_{k=l}^u (x_k- y_k) b^k=0$. But this forces $x_l= y_l$ since $\{b^k\}$ forms a basis. We obtain a contradiction. Therefore, we conclude that $x \mapsto P_x(t)$ is injective.

\;\\
\textbf{Case 2:} The general case $p \in (0,2)$.

The strategy will be similar but we need refined estimates and induction arguments. 

Take $n \in \N$ such that $n> 1+ \frac{p}{2-p}$ and take 
\begin{equation}
    \epsilon \in \left( 0, \frac{(2-p)(n-1)-p}{p+2} \right), \quad \delta= \epsilon+ \frac{2-p}{p+2}.
\end{equation}

Let $x \in \X$. We will inductively construct $P_x^k(t)= P^k(t)$ for $1 \le k \le dN$ with each $P^k(t)$ being a solution of \eqref{eq N-body} for large $t$, and for all $k=1, 2, \cdots, dN$:
\begin{itemize}
    \item if $p_k \ge 0$, then for some $f_k \in S_{p_k- \frac{2-p}{p+2}}$, one has
    \begin{equation} \label{eq induction asymptotic hard} 
        \alpha^k(t)= \alpha^{k-1}(t)+ x_k b^k t^{p_k}+ f_k(t)+ o \big( t^{-\delta+} \big) \quad \text{as} \quad t \to +\infty,
    \end{equation}
    \item if $p_k<0$, then \eqref{eq induction asymptotic} holds,
\end{itemize}
where we set by convention $\alpha^0(t)= at^{\frac{2}{p+2}}$, $\beta^0(t)= \dot{\alpha}^0(t)$. 

Suppose $P^{k-1}(t)$ is constructed for some $k \ge 1$ with required properties. 

If $p_k<0$, then we have $\frac{2-p}{p+2}- p_k>0$, and the proof is identical to that of case 1. In the following we assume $p_k \ge 0$.

Note that $p_k \ge 0$ implies $p_i \ge 0$ for $1 \le i \le k-1$, so \eqref{eq induction asymptotic hard} holds for these $i$. Using $p_k \le \frac{p}{p+2}$ and $p \in (0,2)$, we deduce that there exists $g \in S_{\frac{p}{p+2}}$ such that
\begin{equation}
    \alpha^{k-1}(t)= at^{\frac{2}{p+2}}+ g(t)+ O \big( t^{-\epsilon} \big).
\end{equation}
Let $f_k \in S_{p_k- \frac{2-p}{p+2}}$, $|z(t)| \le t^{-\epsilon}$ be to be determined later, and
\begin{equation}
    \alpha^k(t) = \alpha^{k-1}(t)+ x_k b^k t^{p_k}+ f_k(t)+ z(t).
\end{equation}
Again, we will try to write \eqref{eq N-body} in the form of \eqref{eq ODE1}. However, in this case, the remainders in the Taylor formula will be larger, so we need expansions up to higher orders.

Since $\alpha^k- \alpha^{k-1}= O(t^{p_k})$, by the Taylor formula, we have
\begin{align}
    \nabla U(\alpha^k)- \nabla U(\alpha^{k-1})= \sum_{s=1}^n \frac{1}{s!} \nabla^{s+1} U(\alpha^{k-1}) \big( (\alpha^k- \alpha^{k-1})^{\otimes s} \big)+ O \big( t^{-2-\frac{2n}{p+2}+ (n+1)p_k} \big).
\end{align}
For simplicity, we write $b=x_k b^k$, $q=p_k$ and $f(t)= f_k(t)$. We further have
\begin{align}
    \nabla^{s+1} U(\alpha^{k-1}) \big( (\alpha^k- \alpha^{k-1})^{\otimes s} \big)= \nabla^{s+1} U(at^{\frac{2}{p+2}}+ g(t)) \big( (b t^q+ f(t))^{\otimes s} \big)+ O \big( t^{-2-\frac{2-p}{p+2}- \epsilon} \big)
\end{align}
for $s \ge 2$ and 
\begin{align}
    \nabla^2 U(\alpha^{k-1}) (\alpha^k- \alpha^{k-1}) &= \nabla^2 U(at^{\frac{2}{p+2}}+ g(t)) \big( b t^q+ f(t) \big)+ \frac{Az(t)}{t^2} \\
    &\quad + O \big( t^{-2-\frac{2-p}{p+2}- \epsilon} \big).
\end{align}
Therefore, we obtain
\begin{equation}
    \nabla U(\alpha^k)- \nabla U(\alpha^{k-1})= \frac{Az(t)}{t^2}+ \sum_{s=1}^n \frac{1}{s!} \nabla^{s+1} U(at^{\frac{2}{p+2}}+ g(t)) \big( (b t^q+ f(t))^{\otimes s} \big)+ O \big( t^{-2-\delta} \big).
\end{equation}
Then the equation $\ddot{\alpha}^k= \nabla U(\alpha^k)$ becomes
\begin{equation} \label{eq ODE2}
    \ddot{z}(t)= \frac{Az(t)}{t^2}+ G(f(t))+ O \big( t^{-2-\delta} \big),
\end{equation}
where 
\begin{equation}
    G(f(t))= \sum_{s=1}^n \frac{1}{s!} \nabla^{s+1} U(at^{\frac{2}{p+2}}+ g(t)) \big( (b t^q+ f(t))^{\otimes s} \big)- \frac{\d^2}{\dt^2} \big( b t^q+ f(t) \big).
\end{equation}
By similar calculations, the partial derivative in $z$ of the $O(t^{-2-\delta})$ term is also $O(t^{-2-\delta})$.

\begin{lem} \label{lem claim}
If $g \in S_{\frac{p}{p+2}}$, then there exists $f \in S_{q- \frac{2-p}{p+2}}$ such that 
\begin{equation}
    G(f(t))= O \big( t^{-2-\delta} \big) \quad \text{as} \quad t \to +\infty.
\end{equation}
\end{lem} 

Suppose Lemma \ref{lem claim} is correct for now. Since $\nabla_z G(f(t))=0$, we can apply Lemma \ref{lem ODE} to derive that \eqref{eq ODE2} has a solution $z(t)$, which satisfies $z(t)=o(t^{-\delta+})$ as $t \to +\infty$. Thus, we can construct $P^k$ with required properties.

Finally, we set $P_x(t)= P_x^{dN}(t)$. Then \eqref{eq asymptotic parabolic} is derived from \eqref{eq induction asymptotic hard} and \eqref{eq induction asymptotic}. To prove $x \mapsto P_x(t)$ is injective, note that if $p_k \ge 0$, then $-\delta \le -\frac{2-p}{p+2}+ p_k$, so \eqref{eq induction asymptotic hard} implies 
\begin{equation}
    \alpha^k(t)= \alpha^{k-1}(t)+ x_k b^k t^{p_k}+ o(t^{p_k}).
\end{equation}
Moreover, \eqref{eq induction asymptotic} always implies the above, so we can argue as in case 1.
\end{proof}

\begin{proof}[Proof of Lemma \ref{lem claim}]
By homogeneity of $U$, we have
\begin{equation}
    \nabla^{s+1} U(at^\frac{2}{p+2}+ g(t))= t^{-2-\frac{2(s-1)}{p+2}} \nabla^{s+1} U(a+ t^{-\frac{2}{p+2}} g(t)).
\end{equation}
Since $g \in S_{\frac{p}{p+2}}$, there exist $j \in \N$, $x^i \in \X$, $r_i \ge \frac{2-p}{p+2}>0$ and $k_i \in \N$ such that
\begin{equation}
    t^{-\frac{2}{p+2}} g(t)= \sum_{i=1}^j x^i t^{-r_i} \log^{k_i} (t).
\end{equation}
By applying the Taylor formula for multi-variable functions to $\frac{1}{s!} \nabla^{s+1} U(a+ t^{-\frac{2}{p+2}} g(t))$, with $t^{-r_i} \log^{k_i} (t)$, $1 \le i \le j$ being the variables, we get
\begin{equation}
    \frac{1}{s!} \nabla^{s+1} U(a+ t^{-\frac{2}{p+2}} g(t))= \sum_{i=1}^{j^{(s)}} A_i^{(s)} t^{-r_i^{(s)}} \log^{k_i^{(s)}} (t)+ O \big( t^{-\frac{2}{p+2}- \delta} \big) \quad \text{for } s \ge 2,
\end{equation}
where $j^{(s)} \in \N$, $r_i^{(s)} \ge 0$, $k_i^{(s)} \in \N$, $A_i^{(s)}$ is a $(1,s)$-tensor on $\X$, and
\begin{equation}
    \nabla^2 U(a+ t^{-\frac{2}{p+2}} g(t))= A+ \sum_{i=1}^{j^{(1)}} A_i^{(1)} t^{-r_i^{(1)}} \log^{k_i^{(1)}} (t)+ O \big( t^{-\frac{2}{p+2}- \delta} \big),
\end{equation}
where $j^{(1)} \in \N$, $r_i^{(1)} \ge \frac{2-p}{p+2} $, $k_i^{(1)} \in \N$, $A_i^{(1)}$ is a $(1,1)$-tensor on $\X$. Note that
\begin{equation} \label{eq r_i^s}
    r_i^{(s)}+ \frac{(s-1)(2-p)}{p+2} \ge \frac{2-p}{p+2}>0 \quad \text{for} \;\; s \ge 1,\ 1 \le i \le j^{(s)}.
\end{equation}

Combining these and $\frac{\d^2}{\dt^2} (bt^{p_k})= \frac{Ab t^{p_k}}{t^2}$, if we set
\begin{equation}
    H(f(t))= \sum_{s=1}^n t^{-2-\frac{2(s-1)}{p+2}} \sum_{i=1}^{j^{(s)}} t^{-r_i^{(s)}} \log^{k_i^{(s)}} (t) A_i^{(s)} \big( (b t^q+ f(t))^{\otimes s} \big)+ \frac{A f(t)}{t^2}- \ddot{f}(t),
\end{equation}
then 
\begin{equation} \label{eq G-H}
    G(f(t))= H(f(t))+ O \big( t^{-2-\delta} \big).
\end{equation} 
It then suffices to find $f \in S_{q- \frac{2-p}{p+2}}$ such that $H(f(t))= O \big( t^{-2-\delta} \big)$.

The following lemma is an important ingredient for us to construct $f$.

\begin{lem} \label{lem ODE by algebra}
Let $B \in \R^{dN \times dN}$ be diagonalizable over $\C$. For $\gamma \in \R$, we define
\begin{align}
    T_{\gamma} := \left\{ \sum_{i=1}^m y_i t^{\gamma} \log^{k_i} (t) \ \Big| \ m \in \N,\ y_i \in \X,\ k_i \in \N \right\}.
\end{align}
Then for any $h \in T_{\gamma}$, there exists $f \in T_{\gamma+2}$ such that
\begin{equation} \label{eq ODE by algebra}
    \ddot{f}(t)= \frac{Bf(t)}{t^2}+ h(t).
\end{equation}
\end{lem}

The proof will be given in the appendix. With this lemma, our strategy is to inductively cancel the terms in $H$ with the lowest decay.

Set $F_0(t)=0$. By looking at each term in $H(F_0(t))$, we see $H(F_0(t)) \in S_{\mu_0}$, where
\begin{equation}
    \mu_0 = \max \left\{ -2-\frac{2(s-1)}{p+2}- r_i^{(s)}+ sq \; \middle| \begin{array}{l}
        1 \le s \le n \\ 
        1 \le i \le j^{(s)}
    \end{array} \right\}.
\end{equation}
Since functions in $S_{\mu_0}$ are finitely generated, there exists $h_1 \in T_{\mu_0}$ and $\mu_0'< \mu_0$ such that 
\begin{equation}
    H(F_0(t))- h_1(t) \in S_{\mu_0'}.
\end{equation}
By Lemma \ref{lem ODE by algebra}, there exists $g_1 \in T_{\mu_0+2}$ solving
\begin{equation}
    \ddot{g}_1(t)= \frac{Ag_1(t)}{t^2}+ h_1(t).
\end{equation}

Set $F_1(t)= F_0(t)+ g_1(t)$. We then have $H(F_1(t)) \in S_{\tilde{\mu}_1}$ by looking at each term, where
\begin{equation}
    \tilde{\mu}_1 = \max \Biggl\{ -2-\frac{2(s-1)}{p+2}- r_i^{(s)}+ (s-s')q+ s'(\mu_0+2) \; \Biggm| \begin{array}{l} 
        \scriptstyle 1 \le s \le n \\ 
        \scriptstyle 1 \le i \le j^{(s)} \\ 
        \scriptstyle 0 \le s' \le s 
    \end{array} \Biggr\}.
\end{equation}
However, if $s' \ge 1$, then the quantity inside ``max'' is less than $\mu_0$, the proof of which is identical to \eqref{eq elementary estimate}; if $s'=0$ then the corresponding term is also in $H(F_0(t))$, while such terms are canceled by the choice of $g_1$. This means all the terms in $H(F_1(t))$ that belong to $T_{\mu_0}$ are canceled. Thus we actually have $H(F_1(t)) \in S_{\mu_1}$, where 
\begin{equation}
    \mu_1 = \max \Biggl\{ -2-\frac{2(s-1)}{p+2}- r_i^{(s)}+ (s-s')q+ s'(\mu_0+2) \; \Biggm| \begin{array}{l} 
        \scriptstyle 1 \le s \le n \\ 
        \scriptstyle 1 \le i \le j^{(s)} \\ 
        \scriptstyle 0 \le s' \le s 
    \end{array} \Biggr\} \setminus \left\{ \mu_0 \right\}.
\end{equation}
As before, there exist $h_2 \in T_{\mu_1}$ and $\mu_1'< \mu_1$ such that 
\begin{equation}
    H(F_1(t))- h_2(t) \in S_{\mu_1'}.
\end{equation}
By Lemma \ref{lem ODE by algebra}, there exists $g_2 \in T_{\mu_1+2}$ solving
\begin{equation}
    \ddot{g}_2(t)= \frac{Ag_2(t)}{t^2}+ h_2(t).
\end{equation}

Continuing this process, we can construct $g_1, g_2, \cdots$ with $g_{j+1} \in T_{\mu_j+2}$ such that 
\begin{equation}
    H(F_j(t)) \in S_{\mu_j}, \quad \text{for } F_j:= g_1+ \cdots+ g_j
\end{equation} 
and 
\begin{equation} \label{eq formula for mu}
    \mu_{j+1}= \max \biggl\{ -2-\frac{2(s-1)}{p+2}- r_i^{(s)}+ (s-s')q+ \sum_{v=1}^{s'} (\mu_{j_v}+2) \; \biggm| \begin{array}{l} 
        \scriptstyle 1 \le s \le n, \; 1 \le i \le j^{(s)} \\ 
        \scriptstyle 0 \le s' \le s,\; 0 \le j_1, \cdots, j_{s'} \le j
    \end{array} \biggr\} \setminus \{ \mu_0, \cdots, \mu_j \}.
\end{equation}
We prove that the sequence $\{\mu_j\}$ defined by \eqref{eq formula for mu} is strictly decreasing and tends to $-\infty$.

First, we show that $\{\mu_j\}$ is strictly decreasing. Assume that we know $\mu_0> \cdots> \mu_j$. If each $j_v \le j-1$ in \eqref{eq formula for mu}, then the quantity inside ``max'' is automatically less than or equal to $\mu_j$. If any $j_v$ equals $j$, then by \eqref{eq r_i^s}, we have 
\begin{equation} \label{eq elementary estimate} \begin{aligned}
    &-2-\frac{2(s-1)}{p+2}- r_i^{(s)}+ (s-s')q+ \sum_{v=1}^{s'} (\mu_{j_v}+2) \\
    \le & -2-\frac{2(s-1)}{p+2}- r_i^{(s)}+ (s-s')q+ (\mu_j+2)+ (s'-1)(\mu_0+2) \\
    \le &-\frac{2(s-1)}{p+2}- r_i^{(s)}+ (s-s')q+ \mu_j+ (s'-1)q \\
    = &\ \mu_j- r_i^{(s)}- (s-1) \Big( \frac{2}{p+2}-q \Big)< \mu_j.
\end{aligned} \end{equation} 
We point out that \eqref{eq elementary estimate} actually validates the formula \eqref{eq formula for mu}, as this implies all the terms in $H(F_{j+1}(t))$ that belong to $T_{\mu_j}$ are canceled. The logic is the same as how we derive $\mu_1$ from $\tilde{\mu}_1$. 

We then deduce $\mu_{j+1} \le \mu_j$. But in \eqref{eq formula for mu} we are taking ``max'' over a finite set that does not contain $\mu_j$, so $\mu_{j+1} \neq \mu_j$. We obtain $\mu_{j+1}< \mu_j$. By induction we see $\{\mu_j\}$ is strictly decreasing.

Next, we show that for each $j$, there exists $w_i^{(s)} \in \N$ and $v_1, v_2 \in \N$ such that
\begin{equation}
    \mu_j= v_1 (\mu_0+2)+ v_2 q- (v_1+v_2-1) \frac{2}{p+2}- \sum_{\substack{1 \le s \le n \\ 1 \le i \le j^{(s)}}} w_i^{(s)} r_i^{(s)}-2.
\end{equation}
By taking $v_1=1, v_2=0$ and all $w_i^{(s)}=0$, we see $\mu_0$ takes this form. For $\mu_j$ with $j>0$, we can directly see that the expression inside ``max'' in \eqref{eq formula for mu} preserves this form.

Finally, for any $C \in \R$, assume for some $v_1, v_2, w_i^{(s)} \in \N$ we have
\begin{equation}
    v_1 (\mu_0+2)+ v_2 q- (v_1+v_2-1) \frac{2}{p+2}- \sum_{\substack{1 \le s \le n \\ 1 \le i \le j^{(s)}}} w_i^{(s)} r_i^{(s)}-2 \ge C.
\end{equation}
Since $\mu_0+2< q< \frac{2}{p+2}$, there is an upper bound for $v_1$ and $v_2$. Besides, for $1 \le s \le n$, $1 \le i \le j^{(s)}$ with $r_i^{(s)}>0$, there is also an upper bound for $w_i^{(s)}$, and if $r_i^{(s)}=0$, then it does not contribute to the sum. Therefore, the set
\begin{equation}
    \Bigg\{ v_1 (\mu_0+2)+ v_2 q- (v_1+v_2-1) \frac{2}{p+2}- \sum_{\substack{1 \le s \le n \\ 1 \le i \le j^{(s)}}} w_i^{(s)} r_i^{(s)}-2 \ \Big| \ v_1, v_2,w_i^{(s)} \in \N \Bigg\} \cap [C,+\infty)
\end{equation}
is finite. Then we have $\# \{\mu_j\} \cap [C,+\infty)< \infty$, and thus $\liminf_{j \to +\infty} \mu_j= -\infty$ since $C$ is arbitrary. We conclude $\lim_{j \to +\infty} \mu_j= -\infty$ by monotonicity.

Therefore, there exists $m \in \N$ such that $\mu_m< -2-\delta$. Taking $f=F_m$, then we have $f \in S_{\mu_0+2}$ and $H(f(t))= O(t^{-2-\delta})$. By \eqref{eq r_i^s}, we have $\mu_0+2 \le q- \frac{2-p}{p+2}$, and thus $f \in S_{q- \frac{2-p}{p+2}}$. We then conclude recalling \eqref{eq G-H}.
\end{proof}

\section{Hyperbolic-parabolic solutions}

\begin{proof}[Proof of Theorem \ref{thm mixed}]\

The proof looks like a combination of that of Theorem \ref{thm hyperbolic} and \ref{thm parabolic}. 

We need an analogue of Lemma \ref{lem expansion hyperbolic}, whose proof will be more similar to that of Lemma \ref{lem claim}.

\begin{lem} \label{lem expansion mixed}
There exists $\delta>0$ and $f \in S_{1-p}$ such that for $\tilde{\alpha}(t)= vt+ \alpha^0(t)+ f(t)$, we have 
\begin{equation}
    \ddot{\tilde{\alpha}}- \nabla U(\tilde{\alpha})= O(t^{-2-\delta}) \quad \text{as} \quad t \to +\infty.
\end{equation}
\end{lem}

\begin{proof}
For simplicity, we set $M=M_v$. For $I \neq J \in M$ and $\alpha \in \Y$, we set
\begin{equation}
    U_I(\alpha)= \frac{1}{p} \sum_{i<j \in I} \frac{\lambda_i \lambda_j}{|\alpha_i- \alpha_j|^p} \quad \text{and} \quad U_{IJ}(\alpha)= \frac{1}{p} \sum_{i \in I, j \in J} \frac{\lambda_i \lambda_j}{|\alpha_i- \alpha_j|^p}.
\end{equation}
Then we have
\begin{equation}
    U(\alpha)= \sum_{I \in M} U_I(\alpha)+ \frac{1}{2} \sum_{I \neq J \in M} U_{IJ}(\alpha).
\end{equation}

Since $\frac{2}{p+2}+ p-1>0$, there exists $n \in \N$ such that $n(\frac{2}{p+2}+ p-1)>1$. Take $\delta$ such that
\begin{equation}
    \delta>0, \quad \delta< \Big( \frac{2}{p+2}+p-1 \Big)n+ p-1 \quad \text{and} \quad \delta< \frac{2}{p+2}+p-1-q.
\end{equation}
Let $f \in S_{1-p}$, to be determined later. Because $\frac{2}{p+2}+ p-1>0$, $f(t)$ is of lower order than $vt+ \alpha^0(t)$, so $\tilde{\alpha} \in \Y$. This means we may apply the Taylor formula to $\nabla U(\tilde{\alpha})$. Also, the leading term in $\tilde{\alpha}$ is fixed, so we have uniform estimates for the remainder. We have
\begin{align}
    \nabla U_I(\tilde{\alpha})- \nabla U_I (\alpha^0) &= \sum_{s=1}^n \frac{1}{s!} \nabla^{s+1} U_I (\alpha^0) \big( f(t)^{\otimes s} \big)+ O \big( t^{-2-\frac{2n}{p+2}+ (n+1)(1-p)+} \big) \\
    &= \sum_{s=1}^n \frac{1}{s!} t^{-2- \frac{2(s-1)}{p+2}} \nabla^{s+1} U_I(a) \big( f(t)^{\otimes s} \big)+ O \big( t^{-2-\delta} \big)
\end{align}
by \eqref{eq parabolic assumption}, and
\begin{align}
    &\quad \ \nabla U_{IJ}(\tilde{\alpha})- \nabla U_{IJ} \big( vt+ at^\frac{2}{p+2} \big) \\
    &= \sum_{s=1}^n \frac{1}{s!} \nabla^{s+1} U_{IJ} \big( vt+ at^{\frac{2}{p+2}} \big) \big( (f(t)+ O(t^q))^{\otimes s} \big)+ O \big( t^{-p-n-2+ (n+1)(1-p)+} \big) \\
    &= \sum_{s=1}^n \frac{1}{s!} t^{-1-p-s} \nabla^{s+1} U_{IJ} \big( v+ at^{-\frac{p}{p+2}} \big) \big( f(t)^{\otimes s} \big)+ O \big( t^{-2-\delta} \big).
\end{align}
Using Taylor's formula again, there exist $l \in \N$ and $p_k \ge p+1$, $y^k \in \X$ for $1 \le k \le l$ such that
\begin{equation}
    \frac{1}{2} \sum_{I \neq J \in M} \nabla U_{IJ} \big( vt+ at^\frac{2}{p+2} \big)= \sum_{k=1}^l t^{-p_k} y^k+ O \big( t^{-2- \delta} \big).
\end{equation}
Moreover, there exist $j^{(s)} \in \N$, $\bar{r}_i^{(s)} \ge 0$ and $(1,s)$-tensors $A_i^{(s)}$ on $\X$ such that
\begin{equation}
    \frac{1}{2 s!} \sum_{I \neq J \in M} \nabla^{s+1} U_{IJ} \big( v+ at^{-\frac{p}{p+2}} \big)= \sum_{i=1}^{j^{(s)}} A_i^{(s)} t^{-\bar{r}_i^{(s)}}+ O \big( t^{-1- \delta} \big) \quad \text{for } s \ge 1.
\end{equation}

We take
\begin{align}
    &r_i^{(s)}= \bar{r}_i^{(s)}+ \frac{p(s-1)}{p+2}+p \quad \text{for} \;\; s \ge 1,\ 1 \le i \le j^{(s)}, \\
    &r_0^{(1)}=\frac{2}{p+2}+p-1, \quad A_0^{(1)}=0, && \\
    &r_0^{(s)}=0, \quad A_0^{(s)}= \frac{1}{s!} \sum_{I \in M} \nabla^{s+1} U_I(a) \quad \text{for} \;\; s \ge 2, \\
    &A= \sum_{I \in M} \nabla^2 U_I(a) \in \R^{dN \times dN}. 
\end{align}
Using \eqref{eq sub N-body} and previous expansions, we get
\begin{equation}
    \nabla U(\tilde{\alpha})- \ddot{\tilde{\alpha}}(t) = H(f(t))+ O \big( t^{-2- \delta} \big), 
\end{equation}
where
\begin{equation}
    H(f(t))= \sum_{k=1}^l t^{-p_k} y^k+ \sum_{s=1}^n t^{-2-\frac{2(s-1)}{p+2}} \sum_{i=0}^{j^{(s)}} t^{-r_i^{(s)}} A_i^{(s)} \big( f(t)^{\otimes s} \big)+ \frac{Af(t)}{t^2}- \ddot{f}(t),
\end{equation}
with 
\begin{equation} 
    r_i^{(s)}+ (s-1) \Big( \frac{2}{p+2}+p-1 \Big) \ge \frac{2}{p+2}+p-1 \quad \text{for} \;\; s \ge 1,\ 0 \le i \le j^{(s)}
\end{equation}
as the replacement of \eqref{eq r_i^s}. It suffices to find $f \in S_{1-p}$ such that $H(f(t))= O \big( t^{-2-\delta} \big)$. 

The rest of the proof is almost identical to that of Lemma \ref{lem claim}, so we only give a sketch. First, since $\Lambda A$ is symmetric, using the proof of Lemma \ref{lem linear algebra}, we see $A$ is diagonalizable over $\R$.

Take $F_0(t)=0$. Then we have $H(F_0(t)) \in S_{\mu_0}$, where 
\begin{equation}
    \mu_0= \max \{-p_1, \cdots, -p_l\} \le -1-p.
\end{equation} 
Then there exist $h_1 \in T_{\mu_0}$ and $\mu_0'< \mu_0$ such that 
\begin{equation}
    H(F_0(t))- h_1(t) \in S_{\mu_0'}.
\end{equation}
By Lemma \ref{lem ODE by algebra}, there exists $g_1 \in T_{\mu_0+2}$ solving
\begin{equation}
    \ddot{g}_1(t)= \frac{Ag_1(t)}{t^2}+ h_1(t).
\end{equation} 

Repeating this process as in the proof of Lemma \ref{lem claim}, we find $g_1, g_2, \cdots$ with $g_{j+1} \in T_{\mu_j+2}$ such that
\begin{equation}
    H(F_j(t)) \in S_{\mu_j}, \quad \text{for } F_j:= g_1+ \cdots + g_j
\end{equation} 
and 
\begin{equation}
    \mu_{j+1}= \max \biggl\{ -2-\frac{2(s-1)}{p+2}- r_i^{(s)}+ \sum_{v=1}^s (\mu_{j_v}+2) \biggm| \begin{array}{l}
         \scriptstyle 1 \le s \le n \\ 
         \scriptstyle 0 \le i \le j^{(s)} \\ 
         \scriptstyle 0 \le j_1, \cdots, j_s \le j 
    \end{array} \biggr\} \cup \{-p_1, \cdots, -p_l\} \setminus \left\{ \mu_0, \cdots, \mu_j \right\}.
\end{equation}
Again, we will prove that $\{\mu_j\}$ is strictly decreasing and tends to $-\infty$.

To show strict monotonicity, the critical step is an analogue of \eqref{eq elementary estimate}: if any of $j_v=j$, then
\begin{equation} \begin{aligned}
    &-2-\frac{2(s-1)}{p+2}- r_i^{(s)}+ \sum_{v=1}^s (\mu_{j_v}+2) \\
    \le & -2-\frac{2(s-1)}{p+2}- r_i^{(s)}+ (\mu_j+2)+ (s-1)(\mu_0+2) \\
    \le & -(s-1)\Big( \frac{2}{p+2}+p-1 \Big)- r_i^{(s)}+ \mu_j < \mu_j.
\end{aligned} \end{equation}
Then we can prove by induction that for each $j$, there exist $v_k, w_i^{(s)} \in \N$ such that
\begin{equation}
    \mu_j= \sum_{k=1}^l v_k (2-p_k)- \Big( \sum_{k=1}^l v_k-1 \Big) \frac{2}{p+2}- \sum_{\substack{1 \le s \le n \\ 0 \le i \le j^{(s)}}} w_i^{(s)} r_i^{(s)}-2.
\end{equation}
Since $2-p_k \le 1-p< \frac{2}{p+2}$, the above then implies $\# \{\mu_j\} \cap [C,+\infty)< \infty$ for any $C \in \R$, and finally we deduce $\mu_j \to -\infty$ as $j \to +\infty$.

We then conclude that there exists $m \in \N$ such that $\mu_m< -2-\delta$. Take $f=F_m$. Then we have $f \in S_{\mu_0+2} \subset S_{1-p}$ and $H(f(t))= O \big( t^{-2-\delta} \big)$ as $t \to +\infty$. 
\end{proof}

Go back to the proof of Theorem \ref{thm mixed}. Let $\tilde{P}= (\tilde{\alpha}, \tilde{\beta})$, where $\tilde{\beta}= \dot{\tilde{\alpha}}$ and $\tilde{\alpha}$ is as in Lemma \ref{lem expansion mixed}. 

Let $\epsilon \in (0, \delta)$ and $\alpha(t)= \tilde{\alpha}(t)+ z(t)$, where we assume $|z(t)| \le t^{-\epsilon}$. Then we have
\begin{equation}
    \nabla U(\alpha)- \nabla U(\tilde{\alpha})= \nabla^2 U(\tilde{\alpha}) z(t)+ O \big( t^{-2- \frac{2}{p+2}-2\epsilon} \big).
\end{equation}
Since $\tilde{\alpha}= vt+ at^\frac{2}{p+2}+ O(t^{\max\{q,1-p\}})$, we have
\begin{align}
    \nabla^2 U(\tilde{\alpha}) &= \sum_{I \in M} \nabla^2 U_I(\tilde{\alpha})+ O(t^{-2-p}) \\
    &= t^{-2} \sum_{I \in M} \nabla^2 U_I(a)+ O \big( t^{-2- \min\{p, \frac{2}{p+2}- q, \frac{2}{p+2}+ p-1\}} \big).
\end{align}
Take $\delta$ smaller if necessary (Lemma \ref{lem expansion mixed} will still be valid) and set 
\begin{equation}
    A= \sum_{I \in M} \nabla^2 U_I(a).
\end{equation}
Then we have
\begin{equation}
    \nabla U(\alpha)- \nabla U(\tilde{\alpha})= \frac{Az(t)}{t^2}+ O(t^{-2-\delta}),
\end{equation}
and the partial derivative in $z$ of the $O(t^{-2-\delta})$ term is also $O(t^{-2-\delta})$.

By Lemma \ref{lem expansion mixed}, the equation \eqref{eq N-body tight form} becomes \eqref{eq ODE with decay}, where $F$ satisfies \eqref{eq estimate of F} with $\kappa= \delta$. Since $\Lambda A$ is symmetric, $A$ is diagonalizable, so we can apply Lemma \ref{lem ODE} to find a solution of \eqref{eq N-body} satisfying \eqref{eq asymptotic mixed}.

The expansion for the special case $p>1$ can be deduced immediately.
\end{proof}

\appendix

\section{Proof of Lemma \ref{lem ODE}}

We may work instead on $\C$ by setting $F(t,z)= F(t, \mathrm{Re}(z))$ for $z \in \C \otimes \X \simeq \C^{dN}$ and allowing $z(t)$ to take values in $\C^{dN}$. If the complex counterpart is proved, then the lemma follows by taking the real part. We may also assume $\epsilon>0$ and $\epsilon> \delta- \kappa$ by increasing $\epsilon$.

Since $B$ is diagonalizable over $\C$, let $v_1, \cdots, v_{dN} \in \C^{dN}$ be eigenvectors of $B$ forming a basis of $\C^{dN}$, with eigenvalues $c_1, \cdots, c_{dN} \in \C$, respectively. Let $a_j, b_j$ be the two roots of $\lambda^2-\lambda= c_j$. Write $F(t,z)= \sum_{j=1}^{dN} f_j(t,z) v_j$. Then $f_j$ also satisfies \eqref{eq estimate of F}. 

For $T>1$, define $\B= \left\{ z \in C \big( [T,+\infty), \C^{dN} \big) \ \big| \ {\| z \|}_3 \le 1 \right\}$, where
\begin{equation}
    {\| z \|}_3 := \sup \limits_{t \ge T} t^{\epsilon} |z(t)|.
\end{equation}
For $z \in \B$, we define $\Gamma z$ by
\begin{equation}
    (\Gamma z)(t)= \sum_{j=1}^{dN} G_{a_j,b_j}f_j(t,z) v_j,
\end{equation}
where 
\begin{equation}
    G_{a,b}f(t,z)= \left\{ \begin{aligned}
        &\ \frac{G_a f(t,z)- G_b f(t,z)}{a-b}, && a \neq b, \\
        &\ \frac{\d}{\d a} G_a f(t,z), && a=b,
    \end{aligned} \right.
\end{equation}
and
\begin{equation}
    G_a f(t,z)= \left\{ \begin{aligned}
        &\ t^a \int_T^t \tau^{1-a} f(\tau,z(\tau)) \d \tau, && \mathrm{Re}(a) \le -\delta, \\
        &-t^a \int_t^\infty \tau^{1-a} f(\tau,z(\tau)) \d \tau, && \mathrm{Re}(a) > -\delta,
    \end{aligned} \right. 
\end{equation}
for any $a,b \in \mathbb{C}$, function $f(\cdot,\cdot)$ satisfying \eqref{eq estimate of F} and $z \in \B$. 

Note that, if $f$ satisfies \eqref{eq estimate of F}, then for any $x,y \in \B$, we have
\begin{equation}
    |G_a f(t,x)| \lesssim t^{\Re(a)} \int_T^t \tau^{1-\Re(a)} \tau^{-2-\delta} \d \tau \lesssim t^{-\delta} \log t, \quad \text{if } \ \Re(a) \le -\delta,
\end{equation}
\begin{equation}
    |G_a f(t,x)| \lesssim t^{\Re(a)} \int_t^\infty \tau^{1-\Re(a)} \tau^{-2-\delta} \d \tau \lesssim t^{-\delta}, \quad \text{if } \ \Re(a)> -\delta,
\end{equation}
and, using $\kappa+ \epsilon> \delta$, 
\begin{align}
    |G_a f(t,x)- G_a f(t,y)| &\lesssim t^{\Re(a)} \int_T^t \tau^{1-\Re(a)} \tau^{-2-\kappa} |x(\tau)- y(\tau)| \d \tau \\
    &\lesssim t^{\Re(a)} \int_T^t \tau^{-1-\Re(a)- \kappa- \epsilon} {\| x-y \|}_3 \d \tau \lesssim t^{-\delta} {\| x-y \|}_3, \;\; \text{if} \; \Re(a) \le -\delta,
\end{align}
\begin{align}
    |G_a f(t,x)- G_a f(t,y)| \lesssim t^{\Re(a)} \int_t^\infty \tau^{-1-\Re(a)- \kappa- \epsilon} {\| x-y \|}_3 \d \tau \lesssim t^{-\delta} {\| x-y \|}_3, \;\; \text{if} \; \Re(a)> -\delta.
\end{align}
The same results hold for $\frac{\d}{\d a} G_a f$ with the right hand sides multiplied by $\log t$. 

We then deduce that
\begin{equation}
    {\| \Gamma x \|}_3 \lesssim T^{\epsilon- \delta} \log^2 (T) \quad \text{and} \quad {\| \Gamma x- \Gamma y \|}_3 \le T^{\epsilon- \delta} \log T {\| x-y \|}_3, \quad \forall x,y \in \B.
\end{equation}
Therefore, if $T$ is large enough, then $\Gamma$ maps $\B$ to itself and is a contraction. We can directly check that for $z \in \B$, one has
\begin{equation}
    \frac{\d^2}{\dt^2} (\Gamma z)= \frac{B}{t^2} (\Gamma z)+ F(t,z).
\end{equation}
Thus the fixed point of $\Gamma$ in $\B$, guaranteed by the contraction mapping theorem, solves \eqref{eq ODE with decay}. Moreover, note that $z \in \B$ implies $\Gamma z= o(t^{-\delta+})$, so we have the desired decay.

\begin{rmk}
If $B$ is not diagonalizable over $\C$, then the conclusion is still correct. One can prove this using the fact that the set of diagonalizable matrices over $\C$ is dense and modifying the definition of $\Gamma$ slightly. However, such result is not needed in the paper.
\end{rmk}

\section{Proof of Lemma \ref{lem ODE by algebra}}

By linearity, it suffices to consider the case where $h$ consists of one term. We assume $h(t)= w t^{\gamma} \log^m (t)$ for some $w \in \X$ and $m \in \N$. 

First consider the case $\gamma \neq -\frac{3}{2}$. We take
\begin{equation}
    f(t)= \sum_{j=0}^{m+1} v_j t^{\gamma+2} \log^j (t), \quad v_j \in \X.
\end{equation}
Setting $v_{m+2}= v_{m+3}=0$, then we have
\begin{align}
    \ddot{f}(t) &= \sum_{j=0}^{m+1} v_j \Big( (\gamma+2)(\gamma+1) t^\gamma \log^j(t)+ j(2\gamma+3) t^\gamma \log^{j-1}(t)+ j(j-1) t^\gamma \log^{j-2}(t) \Big) \\
    &= \sum_{j=0}^{m+1} \Big( (\gamma+2)(\gamma+1) v_j+ (j+1)(2\gamma+3) v_{j+1}+ (j+2)(j+1) v_{j+2} \Big) t^\gamma \log^j(t).
\end{align}
Thus, $f$ satisfies \eqref{eq ODE by algebra} if for $j=0, \cdots, m+1$, it holds
\begin{equation} 
    Bv_j= (\gamma+2)(\gamma+1) v_j+ (j+1)(2\gamma+3) v_{j+1}+ (j+2)(j+1) v_{j+2}- \delta_{jm} w.
\end{equation}
where $\delta_{jm}$ is the Kronecker delta. Let $T= B- (\gamma+2)(\gamma+1)$. Then we require
\begin{equation} \label{eq recurrence}
    Tv_j= (j+1)(2\gamma+3) v_{j+1}+ (j+2)(j+1) v_{j+2}- \delta_{jm} w.
\end{equation} 
\\[-8pt]
\textbf{Claim:} for any $v' \in \im T$ and $w' \in \X$, there is $v \in \X$ such that $Tv=v'$ and $v-w' \in \im T$.

In fact, since $v' \in \im T$, there exists $w'' \in \X$ such that $Tw''=v'$. Since $B$ is diagonalizable over $\C$, $T$ is also diagonalizable over $\C$. From linear algebra we know $\ker T= \ker T^2$, which, by the rank-nullity theorem, implies $\ker T \cap \im T= \{0\}$ and thus $\ker T+ \im T= \X$. Then there exists $v'' \in \ker T$ such that $w'- w''- v'' \in \im T$. By choosing $v=w''+ v''$, we get both $Tv=v'$ and $v-w' \in \im T$. \\[-8pt]

Note that $2\gamma+3 \neq 0$. We first apply the claim to $v'=0$ and $w'= \frac{w}{(m+1)(2\gamma+3)}$. There exists $v_{m+1} \in \X$ such that \eqref{eq recurrence} holds for $j=m+1$ and for $j=m$ the equation of $v_m$ is solvable. Then we apply the claim to $v'= (m+1)(2\gamma+3) v_{m+1}-w$ and $w'=-\frac{m+1}{2\gamma+3} v_{m+1}$. We find $v_m$ such that \eqref{eq recurrence} holds for $j=m$ and for $j=m-1$ the equation of $v_{m-1}$ is solvable. Continuing this process, we can find $v_0, \cdots, v_{m+1}$ such that \eqref{eq recurrence} holds for all $j$.

For the case $\gamma= -\frac{3}{2}$, we take $f(t)= \sum \limits_{j=0}^{m+2} v_j t^{\gamma+2} \log^j (t), \ v_j \in \X$ and set $v_{m+3}= v_{m+4}=0$. Then we need that for $j=0, \cdots, m+2$, it holds
\begin{equation}
    Tv_j= (j+2)(j+1)v_{j+2}- \delta_{jm}w.
\end{equation}
Similar to above, we apply the claim to $v'=0$ and $w'= \frac{w}{(m+2)(m+1)}$ to get $v_{m+2}$. Then we apply the claim to $v'=  (m+2)(m+1) v_{m+2}$ and $w'=0$ to solve for $v_m$. We can continue this process until we reach $0$ or $1$, whichever has the same parity as $m$. For $j$ such that $j-m$ is odd, we can simply take $v_j=0$.

In conclusion, we have constructed $f$ for all $\gamma \in \R$ and $h \in T_{\gamma}$.

\bibliographystyle{amsplain}
\bibliography{ref}

\end{document}